\documentclass{article}
\usepackage[utf8]{inputenc}
\usepackage{graphicx}
\usepackage{amsfonts}
\usepackage{amsmath}
\usepackage{amssymb}
\usepackage{fancyhdr}
\usepackage{titlesec}
\usepackage{indentfirst}
\usepackage{booktabs}
\usepackage{verbatim}
\usepackage{color}
\usepackage{amsthm}
\usepackage{subfigure}
\usepackage{mathabx}

\usepackage[notref,notcite]{showkeys}

\usepackage[page,header]{appendix}
\usepackage{titletoc}

\newcommand{\rd}{\,\mathrm{d}}
\numberwithin{equation}{section}
\newtheorem{theorem}{Theorem}[section]
\newtheorem{lemma}[theorem]{Lemma}

\newtheorem{remark}[theorem]{Remark}

\usepackage{mathtools}

\def\bk{{\bf k}}

\def\bx{{\bf x}}
\def\by{{\bf y}}

\def\cM{\mathcal{M}}

\def\cC{\mathcal{C}}

\begin{document}

\title{Wasserstein-infinity stability and mean field limit of discrete interaction energy minimizers}

\author{ Ruiwen Shu\footnote{Department of Mathematics, University of Georgia, Athens, GA 30602 (ruiwen.shu@uga.edu).}}

\maketitle

\abstract{In this paper we give a quantitative stability result for the discrete interaction energy on the multi-dimensional torus, for the periodic Riesz potential. It states that if the number of particles $N$ is large and the discrete interaction energy is low, then the particle distribution is necessarily close to the uniform distribution (i.e., the continuous energy minimizer) in the Wasserstein-infinity distance. As a consequence, we obtain a quantitative mean field limit of interaction energy minimizers in the Wasserstein-infinity distance. The proof is based on the application of the author's previous joint work with J. Wang on the stability of continuous energy minimizer, together with a new mollification trick for the empirical measure in the case of singular interaction potentials.}

\section{Introduction}

In this paper we are concerned with the discrete interaction energy
\begin{equation}\label{EN}
    E_N(\vec{\bx}) = \frac{1}{2N(N-1)}\sum_{i,j=1,\, i\ne j}^NW(\bx_i-\bx_j)\,,
\end{equation}
where $N$ is the total number of particles, $\bx_1,\dots,\bx_N\in \mathbb{T}^d=(\mathbb{R}/\mathbb{Z})^d$, and $\vec{\bx}=(\bx_1,\dots,\bx_N)\in (\mathbb{T}^d)^N$ is the vector of the particle configuration. $W:\mathbb{T}^d\rightarrow (-\infty,\infty]$ is the interaction potential, 
assumed to be locally integrable, lower-semicontinuous and bounded from below. 

As $N$ becomes large, by taking the empirical measure
\begin{equation}
\rho_N[\vec{\bx}] = \frac{1}{N}\sum_{i=1}^N \delta(\cdot-\bx_i)\,,
\end{equation}
(where $\delta$ stands for the Dirac delta function on $\mathbb{T}^d$), one intuitively has the mean field limit: the discrete energy $E_N$ should approach the continuous energy
\begin{equation}
E[\rho] = \frac{1}{2}\int_{\mathbb{T}^d}\int_{\mathbb{T}^d}W(\bx-\by)\rho(\by)\rd{\by}\rho(\bx)\rd{\bx}\,,
\end{equation}
where $\rho\in \cM(\mathbb{T}^d)$, a probability measure on $\mathbb{T}^d$. Using the compactness of the underlying space $\mathbb{T}^d$ and the previously stated assumptions on $W$, it is straightforward to prove the existence of minimizers for $E_N$ and $E$. 

The following classical lemma gives the basic relations between $E_N$ and $E$. Its proof can be found in \cite[Section 4.2]{saffbook2} (with a slightly different underlying space).
\begin{lemma}\label{lem_min}
Assume $W$ is locally integrable, lower-semicontinuous and bounded from below. If $N_1<N_2$ then $\min E_{N_1}\le \min E_{N_2}$. Also, for any $N$, $\min E_N \le \min E$, and $\lim_{N\rightarrow\infty}\min E_N = \min E$.
\end{lemma}
For the sake of completeness, we include its proof in Appendix \ref{app_min}.

In the past decade, people have studied the discrete and continuous interaction energies intensively due to its importance in physics and biological and social sciences. The interaction potential may come from physics, for example, the Coulomb interaction or the Lennard-Jones potential, or from the study of collective behavior of social agents, where the potential is often repulsive at short distance and attractive at long distance. The energy minimizers can exhibit interesting pattern formation, including ring patterns, clustering and fractals. Many existing results are concerned with the existence, uniqueness and qualitative behavior of minimizers \cite{Ca13,Ca13_2,CCP15,SST15,Lop19,ShuCar,BCT,BKSUB,CDM,CFP17,KSUB,ST,Fra,DLM1,DLM2,FM}.

This paper will study the following questions:
\begin{itemize}
\item Stability of discrete energy minimizers: suppose we know that $\vec{\bx}$ satisfies that $E_N(\vec{\bx})$ is sufficiently close to its minimum, can we obtain that $\vec{\bx}$ is close to a minimizer in certain metric?
\item Mean field limit of discrete energy minimizers: for each $N$, let $\vec{\bx}_{(N)}$ be a minimizer of $E_N$. Then is it true that $\{\rho_N[\vec{\bx}_{(N)}]\}$ converges weakly to a minimizer of $E$ as $N\rightarrow\infty$? If yes, can one quantify the rate of convergence in certain metric?
\end{itemize}



In this paper we restrict our attention to the periodic Riesz potential with parameter $s<d$, defined by its Fourier coefficients\footnote{Here the Fourier coefficient of a function $f\in L^1(\mathbb{T}^d)$ is defined by $\hat{f}(\bk)=\int_{\mathbb{T}^d} f(\bx)e^{-2\pi i \bk\cdot\bx}\rd{\bx}$ for $\bk\in\mathbb{Z}^d$.}
\begin{equation}
\hat{W}_s(\bk) = \left\{\begin{array}{ll} |\bk|^{-d+s},& \quad \textnormal{if }\bk\ne 0 \\  0,&\quad \textnormal{if }\bk=0\end{array}\right.\,.
\end{equation}
Some basic properties of $W_s$ are outlined in Lemma \ref{lem_period} below. Compared with the possibly complicated behavior of the minimizers of $E_N$, it is easy to see that the minimizer of $E$ is unique, being the uniform distribution $\rho=1$ with the minimal energy being $E[1]=0$. In fact, this follows from 
\begin{equation}
E[\rho] = \frac{1}{2}\sum_{\bk} \hat{W}_s(\bk)|\hat{\rho}(\bk)|^2 = \frac{1}{2}\sum_{\bk\ne 0} |\bk|^{-d+s}|\hat{\rho}(\bk)|^2\,,
\end{equation}
(see \cite[Appendix 2]{SW21} for a rigorous justification of this equation). Furthermore, \cite{SW21GET} gives a stability result of this energy minimizer, as a multi-dimensional generalization of the classical Erd\H{o}s-Tur\'an inequality for polynomial root distribution \cite{ET,SW21}:
\begin{theorem}[{\cite[Theorem 1.3]{SW21GET}}]\label{thm_GET}
	Let $W$ be the periodic Riesz potential with parameter $s$. Then any $\rho\in\cM(\mathbb{T}^d)$ satisfies:
	\begin{enumerate}
		\item[\textnormal{(i)}]  If $d\ge 2$, then (denoting $d_\infty$ as the Wasserstein-infinity distance on $\mathbb{T}^d$)
		\begin{equation}\label{thm_GET_1}
			d_\infty(\rho,1) \lesssim  E[\rho]^\gamma,\quad \gamma = \frac{1}{2d-s}.
		\end{equation}
		\item[\textnormal{(ii)}] If $d=1$, then
		\begin{enumerate}
			\item If $s < 0$, then \eqref{thm_GET_1} also holds.
			\item If $s = 0$, then
			\begin{equation}\label{thm_GET_2}
				d_\infty(\rho,1)(1+|\log d_\infty(\rho,1)|)^{-1/2} \lesssim  E[\rho]^{1/2}.
			\end{equation}
			\item If $ 0<s<1$, then
			\begin{equation}
				d_\infty(\rho,1)  \lesssim  E[\rho]^{1/2}.
			\end{equation}
		\end{enumerate}
	    	\item[\textnormal{(iii)}]   \textnormal{(i)} and \textnormal{(ii)} are sharp up to the constants with the exception $d = 1$ and $s = 0$.
	\end{enumerate}

\end{theorem}
This is a quantitative result, stating that if a probability measure $\rho$ has energy close to the minimal energy, then it is necessarily close to the unique energy minimizer in the Wasserstein-infinity distance.

Due to the close relation between the discrete and continuous energies, one expects to have the following result:
\begin{quote}
If $N$ is large and a configuration vector $\vec{\bx}$ has its discrete energy $E_N[\vec{\bx}]$ close to the (continuous) minimal energy, then the corresponding empirical measure $\rho_N[\vec{\bx}]$ is necessarily close to the (continuous) energy minimizer.
\end{quote}
Our main result gives a quantitative conclusion in this direction. Within this paper, all constants $c,C$ (possibly with subscripts) are positive, depending on $d$ and $s$, and does not depend on other parameters unless stated otherwise. Each $c,C$ may refer to a different constant.
\begin{theorem}\label{thm2}
Let $W$ be the periodic Riesz potential with parameter $s<d$. Let $\vec{\bx}\in(\mathbb{T}^d)^N$. Then
\begin{equation}\label{thm2_1}
    d_\infty(\rho_N[\vec{\bx}],1) \le C_{(1)}(E_N(\vec{\bx}) + C_{(2)} N^{-\lambda})^\gamma\,,
\end{equation}
where $\lambda>0$ depends on $d,s$, and
\begin{equation}\label{thm2_2}
\gamma = \left\{\begin{split}
& \frac{1}{2d-s},\quad d\ge 2; \text{ or }d=1,\,s<0 \\
& \frac{1}{2},\quad d=1,\,0<s<1 \\
& \frac{1}{2}+\epsilon,\quad d=1,\,s=0
\end{split}\right.\,.
\end{equation}
Here, in the case $d=1,\,s=0$, $\epsilon$ can be taken as any small positive number with $C_{(1)}$ depending on $\epsilon$.
\end{theorem}

In the case $d=1,\,s=0$ one can get a better inequality with $\gamma=\frac{1}{2}$ and an extra logarithmic factor (following the corresponding case in Theorem \ref{thm_GET}), but we do not state it here for the sake of conciseness.

\begin{remark}
Theorem \ref{thm2} also works for small smooth perturbations of the periodic Riesz potential. This is to say, if one fixes a smooth perturbation profile $\Phi(\bx)$ on $\mathbb{T}^d$ with $\Phi(\bx)=\Phi(-\bx)$ and $\int_{\mathbb{T}^d}\Phi(\bx)\rd{\bx}=0$, then Theorem \ref{thm2} works for $W(\bx)=W_s(\bx)+\epsilon \Phi(\bx)$ for all sufficiently small $|\epsilon|$. To see this, we first recall from \cite[Theorem 1.3]{SW21GET} that Theorem \ref{thm_GET} works for such $W$ because the positive-definite condition $\hat{W}(\bk) \ge c|\bk|^{-d+s},\,\forall \bk\ne 0$ is satisfied for sufficiently small $|\epsilon|$ (since $\hat{\Phi}$ has rapid decay), and the regularity assumptions therein on $W$ are satisfied. Besides the application of Theorem \ref{thm_GET}, the proof of Theorem \ref{thm2} in this paper only uses the properties of the potential $W$ as stated in Lemmas \ref{lem_period}, \ref{lem_W} and \ref{lem_Wlog}. It is clear that all these properties are still true in the presence of a smooth perturbation.
\end{remark}

\begin{remark}
The exponent $\lambda$ can be taken as 1 when $s<0$. When $0<s<d$, the choice of $\lambda$ follows Theorem \ref{thm1}, which can be calculated explicitly as outlined in Remark \ref{rem_lambda}. When $s=0$, one can take $\lambda=1-\epsilon$ for arbitrarily small $\epsilon$ (following Theorem \ref{thm1log}, a variant of Theorem \ref{thm1}). 
\end{remark}

\begin{remark}
We do not expect the exponent $\lambda$ to be sharp. However, we do know that $\lambda$ cannot be greater than $\frac{1}{d\gamma}$ in such an inequality. In fact, taking $\vec{\bx}$ to be a minimizer of $E_N$, then Lemma \ref{lem_min} gives $E_N(\vec{\bx}) = \min E_N \le \min E = 0$. Therefore the RHS of \eqref{thm2_1} satisfies $C(E_N(\vec{\bx}) + C N^{-\lambda})^\gamma \le C N^{-\lambda\gamma}$. The LHS of \eqref{thm2_1} is at least $c N^{-1/d}$. In fact, the definition of $d_\infty$ shows that for $\vec{\bx} = (\bx_1,\dots,\bx_N)$, the balls centered at $\bx_1,\dots,\bx_N$ with radius $d_\infty(\rho_N[\vec{\bx}],1)$ necessarily covers $\mathbb{T}^d$. Therefore $d_\infty(\rho_N[\vec{\bx}],1)^d |B(0;1)| \cdot N \ge 1$, leading to $d_\infty(\rho_N[\vec{\bx}],1) \ge c N^{-1/d}$. Combining with \eqref{thm2_1}, we see that $c N^{-1/d} \le C N^{-\lambda\gamma}$ (for any $N$), which implies that $\lambda \le \frac{1}{d\gamma}$.
\end{remark}


The mean field limit of energy minimizers is known to be true under very mild assumptions of $W$, up to taking a subsequence (this is necessary because the continuous energy minimizer may not be unique for general $W$). This has been justified on compact underlying sets \cite[Theorem 4.2.2]{saffbook2} and the whole Euclidean space \cite{CP18}. For the periodic Riesz potential with $0\le s < d$, the minimal energy of $E_N$ has been quantified by \cite{HSSS} as 
\begin{equation}
\min E_N = \left\{\begin{array}{ll}
-C N^{-1+s/d} + o(N^{-1+s/d}), & 0<s<d \\
-C N^{-1} \ln N + \tilde{C} N^{-1} + o(N^{-1}), & s=0
\end{array}\right.
\end{equation}
where $C>0$ and $\tilde{C}\in\mathbb{R}$ are constants depending on $s$. 
For the quantitative convergence of discrete minimizers to the continuous minimizer, if $d-2\le s <d$ (Newtonian or more singular potentials), the techniques developed for the modulated energy in \cite{LS18,Ser20b} can be applied to obtain a quantitative mean field limit in negative Sobolev spaces. Similar results can be obtained for $0\le s < d-2$ using the methods in \cite{NRS}. In the case $s=d-2$, \cite{CHM} obtained a quantitative mean field limit in the Wasserstein-1 distance for whole-space minimizers with confining potentials. \cite{LS} gives a quantitative `regularity' estimate for the discrete minimizer in terms of the size of the Fourier coefficients of the empirical measure. However, we are not aware of any existing results which quantify the distance between discrete and continuous minimizers in the Wasserstein-$p$ metric other than $p=1$.

In this paper we apply our main result, Theorem \ref{thm2}, to obtain the following quantitative mean field limit.
\begin{theorem}\label{thm3}
Let $W$ be the periodic Riesz potential with parameter $s<d$. Let $\vec{\bx}$ be a minimizer of $E_N$. Then
\begin{equation}
    d_\infty(\rho_N[\vec{\bx}],1) \le C N^{-\lambda \gamma}\,,
\end{equation}
where $\gamma,\lambda$ are the same as in Theorem \ref{thm2}. 
\end{theorem}
In particular, if for each $N$, $\vec{\bx}_{(N)}$ is a minimizer of $E_N$, then we have the quantitative convergence of $\rho_N[\vec{\bx}_{(N)}]$ to the continuous minimizer 1 in the sense that $d_\infty(\rho_N[\vec{\bx}_{(N)}],1) \le C N^{-\lambda \gamma}$. 


\begin{proof}[Proof of Theorem \ref{thm3}]
%
Since $\min E=0$ (achieved by the uniform distribution), we get $\min E_N \le 0$ from Lemma \ref{lem_min}. Let $\vec{\bx}$ be a minimizer of $E_N$. Then $E_N(\vec{\bx})\le 0$, and it follows from Theorem \ref{thm2} that 
\begin{equation}
    d_\infty(\rho_N[\vec{\bx}],1) \le C(E_N(\vec{\bx}) + C N^{-\lambda})^\gamma  \le C N^{-\lambda \gamma}\,.
\end{equation}
\end{proof}

An immediate consequence of Theorem \ref{thm3} is that any particle in a minimizer configuration cannot be too far from its closest neighbor.
\begin{theorem}
Let $W$ be the periodic Riesz potential with parameter $s<d$. Let $\vec{\bx}=(\bx_1,\dots,\bx_N)$ be a minimizer of $E_N$. Then, for any $i=1,\dots,N$,
\begin{equation}
    \min_{j\in \{1,\dots,N\},\,j\ne i}|\bx_i-\bx_j| \le CN^{-\lambda \gamma}
\end{equation}
where $\gamma,\lambda$ are the same as in Theorem \ref{thm2}. 
\end{theorem}
\begin{proof}
For a fixed $i$, denote $r:=\min_{j\in \{1,\dots,N\},\,j\ne i}|\bx_i-\bx_j|$. Then $B(\bx_i,r)$ contains no particle other than $\bx_i$. Therefore, letting $\bx_*$ be a point such that $|\bx_i-\bx_*|=r/2$, we see that $B(\bx_*,r/2)$ contains no particle. By the definition of the Wasserstein-infinity distance, we have $d_\infty(\rho_N[\vec{\bx}],1) \ge r/2$. Therefore the conclusion follows from Theorem \ref{thm3}
\end{proof}

The rest of this paper is organized as follows: in Section \ref{sec_thm2} we prove Theorem \ref{thm2} based on a key intermediate result, Theorem \ref{thm1}, for singular potentials. In Section \ref{sec_Ws} we give some basic properties of the periodic Riesz potential $W_s$ and its mollifications. In Section \ref{sec_thm1s0d} we prove Theorem \ref{thm1} for $0<s<d$. In Section \ref{sec_thm1s0} we prove Theorem \ref{thm1} for $s=0$, i.e., the case of logarithmic potential. 

\section{Proof of Theorem \ref{thm2}}\label{sec_thm2}

In this section we give the proof of Theorem \ref{thm2}, in which the proof of the key intermediate result, Theorem \ref{thm1}, is delayed to Sections \ref{sec_thm1s0d} and \ref{sec_thm1s0}.

For periodic Riesz potentials $W_s$ with $s<0$, $W_s$ is a continuous function on $\mathbb{T}^d$ (see Lemma \ref{lem_period} below), and Theorem \ref{thm2} is a straightforward consequence of Theorem \ref{thm_GET}, with $\lambda=1$.

\begin{proof}[Proof of Theorem \ref{thm2} when $s<0$]
\begin{equation}
E_N(\vec{\bx}) = \frac{1}{2N(N-1)}\left(2N^2 E[\rho_N[\vec{\bx}]] - \sum_{i=1}^NW_s(0)\right) = \frac{N}{N-1}E[\rho_N[\vec{\bx}]] - \frac{1}{2(N-1)}W_s(0)\,.
\end{equation}
By applying Theorem \ref{thm_GET} to $\rho_N[\vec{\bx}]$, we get
\begin{equation}
d_\infty(\rho_N[\vec{\bx}],1) \le C E[\rho_N[\vec{\bx}]]^\gamma = C\Big(\frac{N-1}{N}E_N(\vec{\bx}) + \frac{1}{2N}W_s(0) \Big)^\gamma \le C(E_N(\vec{\bx}) + CN^{-1} )^\gamma\,,
\end{equation}
where $\gamma$ is given by \eqref{thm2_2}. Here notice that the sign of $E_N(\vec{\bx})$  is unknown, and the last inequality is true due to the fact that the fact that $\frac{1}{2}\le \frac{N-1}{N} \le 1$.
\end{proof}

However, this proof does not work for $0\le s < d$ because $W_s(0)=\infty$. As a consequence, $E[\rho_N]=\infty$ for any empirical measure $\rho_N$, and no information can be obtained by applying Theorem \ref{thm_GET} to $\rho_N$. To overcome this difficulty and prove Theorem \ref{thm2} when $0\le s < d$, the key step is to replace an empirical measure $\rho_N$ by a suitable approximation, as stated below.
\begin{theorem}\label{thm1}
Let $W$ be the periodic Riesz potential with parameter $s\in [0,d)$. Let $\vec{\bx}\in(\mathbb{T}^d)^N$. Then there exists a probability measure $\rho$ on $\mathbb{T}^d$, such that
\begin{equation}
    d_\infty(\rho,\rho_N[\vec{\bx}]) \le  N^{-\lambda}\,,
\end{equation}
and 
\begin{equation}
    E[\rho] - E_N(\vec{\bx}) \le C N^{-\lambda}\max\{E_N(\vec{\bx}),1\}\,,
\end{equation}
where $\lambda>0$ depends on $d,s$.
\end{theorem}
Here $\rho$ is an approximation of the empirical measure $\rho_N[\vec{\bx}]$, with the $d_\infty$-distance small, and the possible increment in energy $E[\rho]-E_N(\vec{\bx})$ also small. Assuming Theorem \ref{thm1}, we can prove Theorem \ref{thm2} when $0\le s < d$ as follows.
\begin{proof}[Proof of Theorem \ref{thm2} when $0\le s < d$]
Since $d_\infty(\rho_N[\vec{\bx}],1) \le \sqrt{d}$ always holds, we may assume $E_N(\vec{\bx}) \le 1$ without loss of generality. Let $\rho$ as given in Theorem \ref{thm1}. Then Theorem \ref{thm_GET} gives
\begin{equation}
    d_\infty(\rho,1) \le CE[\rho]^\gamma \le C(E_N(\vec{\bx}) + C N^{-\lambda})^\gamma\,,
\end{equation}
where $\gamma$ is given by \eqref{thm2_2} (here for the case $d=1,\,s=0$ one can easily convert the logarithmic factor in \eqref{thm_GET_2} into a small power). Thus
\begin{equation}\begin{split}
    d_\infty(\rho_N[\vec{\bx}],1) \le & d_\infty(\rho,1) + d_\infty(\rho_N[\vec{\bx}],\rho)  \le  C(E_N(\vec{\bx}) + C N^{-\lambda})^\gamma + N^{-\lambda}\\
     \le &  C(E_N(\vec{\bx}) + C N^{-\lambda})^\gamma \,,
\end{split}\end{equation}
where $\gamma<1$ is used in the last inequality.
\end{proof}

The proof of Theorem \ref{thm1} (in Sections \ref{sec_thm1s0d} and \ref{sec_thm1s0}) is the most technical part in this paper. Notice that in the formal expression $E[\rho_N[\vec{\bx}]]=\frac{1}{2N^2}\sum_{i,j}W(\bx_i-\bx_j)$, the major difference from $E_N(\vec{\bx})$ is the \emph{diagonal} terms, i.e., those with $i=j$. When $W(0)=\infty$, it is exactly these terms that make $E[\rho_N[\vec{\bx}]]$ infinite. The construction of $\rho$ is simply a mollification of the empirical measure $\rho_N[\vec{\bx}]$ to make the diagonal terms behave nicer, but the mollification radius $\epsilon\ll 1$ has to been chosen very carefully:
\begin{itemize}
\item If $\epsilon$ is too small, then the diagonal terms may still be too large.
\item If $\epsilon$ is too large, then $\rho$ may look very different from $\rho_N[\vec{\bx}]$, and the change in off-diagonal terms may get out of control.
\end{itemize}
To choose the correct $\epsilon$, we take a dyadic sequence of candidates $\epsilon_1,\dots,\epsilon_K$. For each $\epsilon_k$, we need to analyze the effect of the mollification on the $(i,j)$ term in the energy, according to the relative size between $\epsilon_k$ and $|\bx_i-\bx_j|$. Then we manage to show that at least one of the candidates will work.

\section{Basic properties of the periodic Riesz potential $W_s$}\label{sec_Ws}

In this section we give some basic properties of the periodic Riesz potential $W_s$ and its mollifications. Recall the singularity structure of $W_s$.
\begin{lemma}\cite[Lemma 9.1]{SW21GET}\label{lem_period}
For any real $s<d$, the Riesz potential $W_s$ is smooth away from 0, and identifying $\mathbb{T}^d$ with $[-1/2,1/2)^d$, we have
\begin{itemize}
    \item If $s\ne 0,-2,-4,\dots$, then  
    \begin{equation}
        W_s(\bx)-a_s|\bx|^{-s},\quad a_s=\pi^{d/2-s}\frac{\Gamma(s/2)}{\Gamma((d-s)/2)}
    \end{equation}
    is smooth near 0.
    \item If $s= -2n,\,n\in\mathbb{Z}_{\ge 0}$, then 
    \begin{equation}
        W_s(\bx)-\pi^{d/2-s}\frac{1}{\Gamma((d-s)/2)}\Big(\frac{2(-1)^{n+1}}{n!} |\bx|^{-s}\ln|\bx|+\frac{(-1)^n}{n!}\sum_{k=1}^n \frac{1}{k}|\bx|^{-s}\Big)
    \end{equation}
    is smooth near 0.
\end{itemize}
\end{lemma}

 We apply it to estimate a mollified version of $W_s$. Fix a radially decreasing mollifier $\phi$ supported on $B(0;1/2)$ with $\int_{\mathbb{R}^d} \phi\rd{\bx} = 1$, and define $\psi=\phi*\phi$ which is also a radially decreasing mollifier, supported on $B(0;1)$ with $\int_{\mathbb{R}^d} \psi\rd{\bx} = 1$. Define $\psi_\epsilon(\bx) = \frac{1}{\epsilon^d}\psi(\frac{\bx}{\epsilon})$ (which can be viewed as a function on $\mathbb{T}^d$, identified with $[-1/2,1/2)^d$, for small $\epsilon$), and 
\begin{equation}\label{Weps}
W^\epsilon_s = W_s*\psi_\epsilon\,.
\end{equation}

From now on, for simplicity, we will suppress the dependence of $W_s$ or $W_s^\epsilon$ on $s$.

Then we give the consequences of the singularity structure of $W$ on its mollification $W^\epsilon$. We will separate into the cases of power-law singularity $0<s<d$ and logarithmic singularity $s=0$. 
\begin{lemma}\label{lem_W}
Assume $0<s<d$. We have the following estimates for $W^\epsilon$ if $\epsilon$ is sufficiently small:
\begin{itemize}
    \item ({\bf W1}) $\max W^\epsilon \le C_1 a_s\epsilon^{-s}$.
    \item ({\bf W2}) $W^\epsilon(\bx) \le C_2 W(\bx)$ for any $|\bx|$ sufficiently small.
    \item ({\bf W3}) $W^\epsilon(\bx) \le W(\bx) + C_3\epsilon^2 a_s|\bx|^{-s-2}$ for any $|\bx|\ge 2\epsilon$ (identifying $\mathbb{T}^d$ with $[-1/2,1/2)^d$).
\end{itemize}
\end{lemma}

Notice that $a_s>0$ for $0<s<d$.

\begin{proof}

To prove ({\bf W1}), we only need to treat the singularity $|\bx|^{-s}$ since smooth parts only contribute $O(1)$ to $\max W^\epsilon$. The maximum of $|\bx|^{-s}*\psi_\epsilon$ (on $\mathbb{R}^d$) is obtained at 0 since both $|\bx|^{-s}$ and $\psi_\epsilon$ are radially decreasing. Therefore
\begin{equation}\begin{split}
    \max (|\bx|^{-s}*\psi_\epsilon) = & |S^{d-1}|\int_0^\epsilon \psi_\epsilon(r)r^{-s+d-1}\rd{r} = |S^{d-1}|\int_0^\epsilon \epsilon^{-d}\psi(\frac{r}{\epsilon})r^{-s+d-1}\rd{r} \\
    = & \epsilon^{-s}|S^{d-1}|\int_0^1 \psi(r)r^{-s+d-1}\rd{r} \le C \epsilon^{-s}\,,
\end{split}\end{equation}
since the last integral is finite for $0<s<d$.

To see ({\bf W3}), notice that
\begin{equation}
    W^\epsilon(\bx)-W(\bx) = \frac{1}{\epsilon^d}\int_{B(\bx;\epsilon)} (W(\by)-W(\bx))\psi(\frac{\by-\bx}{\epsilon})\rd{\by}\,.
\end{equation}
By Taylor expansion,
\begin{equation}
    W(\by)-W(\bx) = (\by-\bx)\cdot \nabla W(\bx) + O(\epsilon^2\sup_{B(\bx;\epsilon)}|\nabla^2 W|),\quad \forall \by\in B(\bx;\epsilon)\,,
\end{equation}
and the linear term does not contribute to the previous integral. The singularity structure of $W$ implies that 
\begin{equation}
    \sup_{B(\bx;\epsilon)}|\nabla^2 W| \le C|\bx|^{-s-2},\quad \forall |\bx|\ge 2\epsilon\,.
\end{equation}
Then ({\bf W3}) follows.

To get ({\bf W2}), we apply ({\bf W1}) for $|\bx|<2\epsilon$ and ({\bf W3}) for $|\bx|\ge 2\epsilon$, utilizing the fact that $W(\bx) \sim a_s |\bx|^{-s}$ for $|\bx|$ small.

\end{proof}

Abusing notation, we will write the conclusion of Lemma \ref{lem_period} with $s=0$ as that $W(\bx)+a_0\ln |\bx|$ is smooth near 0, where $a_0=\pi^{d/2}\frac{2}{\Gamma(d/2)}>0$. 
\begin{lemma}\label{lem_Wlog}
Assume $s=0$. We have the following estimates for $W^\epsilon$ if $\epsilon$ is sufficiently small:
\begin{itemize}
    \item ({\bf W1}) $\max W^\epsilon \le -a_0\ln \epsilon + C_1$.
    \item ({\bf W2}) $W^\epsilon(\bx) \le W(\bx) + C_2a_0$ for any $|\bx|$ sufficiently small.
    \item ({\bf W3}) $W^\epsilon(\bx) \le W(\bx) + C_3\epsilon^2a_0|\bx|^{-2}$ for any $|\bx|\ge 2\epsilon$ (identifying $\mathbb{T}^d$ with $[-1/2,1/2)^d$).
\end{itemize}
\end{lemma}

\begin{proof}
To prove ({\bf W1}), we proceed similarly as Lemma \ref{lem_W}, and we only need to estimate
\begin{equation}\begin{split}
    \max ((-\ln|\bx|)*\psi_\epsilon) = & |S^{d-1}|\int_0^\epsilon \psi_\epsilon(r)r^{d-1}(-\ln r)\rd{r} \\
    = & |S^{d-1}|\int_0^\epsilon \epsilon^{-d}\psi(\frac{r}{\epsilon})r^{d-1}(-\ln r)\rd{r} \\
    = & |S^{d-1}|\int_0^1 \psi(r)r^{d-1}(-\ln (r\epsilon))\rd{r} \\
    = &  |S^{d-1}|\left(  (-\ln \epsilon)\int_0^1 \psi(r)r^{d-1}\rd{r} +\int_0^1 \psi(r)r^{d-1}(-\ln r)\rd{r} \right) \\
    = & (-\ln \epsilon)+ |S^{d-1}|\int_0^1 \psi(r)r^{d-1}(-\ln r)\rd{r} = (-\ln \epsilon) + C\,,
\end{split}\end{equation}
where the second last equality uses the fact that $\int_{\mathbb{R}^d} \psi(\bx)\rd{\bx}=1$.

The proof of ({\bf W3}) and ({\bf W2}) is identical to the counterpart in Lemma \ref{lem_W}.

\end{proof}

\section{Proof of Theorem \ref{thm1} when $0<s<d$}\label{sec_thm1s0d}

We follow the notations in the previous section.

{\bf STEP 1}: mollify $\rho_N$ and set up candidates for the mollification radius $\epsilon$.

 We aim to find a proper $N^{-1/(2s)}\le \epsilon\le N^{-\lambda}$ (where the small parameter $\lambda>0$ will be chosen at the end of the proof), and take $\rho=\phi_\epsilon*\rho_N$. Then, recalling $W^\epsilon$ from \eqref{Weps}, we have
\begin{equation}\begin{split}
    E[\rho] = & \frac{1}{2}\int_{\mathbb{T}^d} (W*\phi_\epsilon*\rho_N)\cdot (\phi_\epsilon*\rho_N)\rd{\bx} =  \frac{1}{2}\int_{\mathbb{T}^d} (W*\phi_\epsilon*\phi_\epsilon*\rho_N)\cdot \rho_N\rd{\bx}\\
    = & \frac{1}{2}\int_{\mathbb{T}^d} (W^\epsilon*\rho_N)\cdot \rho_N\rd{\bx} = E^\epsilon[\rho_N]\,,
\end{split}\end{equation}
where $E^\epsilon$ denotes the energy associated to $W^\epsilon$. Clearly 
$d_\infty(\rho,\rho_N) \le \frac{1}{2}\epsilon \le  N^{-\lambda}$ since the mollification by $\phi_\epsilon$ gives a transport plan from $\rho_N$ to $\rho$ with $d_\infty$ cost $\frac{1}{2}\epsilon$. Also, since $\epsilon\ge N^{-1/(2s)}$, we have
\begin{equation}\begin{split}
    E^\epsilon[\rho_N] - E_N^\epsilon(\vec{\bx}) = & \frac{1}{2N^2} \sum_{i,j} W^\epsilon(\bx_i-\bx_j) - \frac{1}{2N(N-1)} \sum_{i,j:\,i\ne j} W^\epsilon(\bx_i-\bx_j) \\
    = & \frac{1}{2N^2} \sum_i W^\epsilon(0) + (\frac{1}{2N^2}-\frac{1}{2N(N-1)})\sum_{i,j:\,i\ne j} W^\epsilon(\bx_i-\bx_j) \\
    \le & \frac{1}{2N}\max W^\epsilon + \frac{1}{2N^2(N-1)}\cdot N(N-1)\max |W^\epsilon| \\
     \le &  C N^{-1} \epsilon^{-s} \le C N^{-1/2}\,,
\end{split}\end{equation}
using ({\bf W1}).

Therefore, it suffices to prove that
\begin{equation}\label{toprove}
    E_N^\epsilon(\vec{\bx})  - E_N(\vec{\bx}) \le C N^{-\lambda}\max\{E_N(\vec{\bx}),1\}\,,
\end{equation}
for some $\epsilon$ satisfying 
$N^{-1/(2s)}\le\epsilon\le N^{-\lambda}
$.

Now we fix $\epsilon_0 = N^{-\lambda}$ and a constant $A\ge 4$ to be chosen independent of $N$. Define 
\begin{equation}
    \epsilon_k = \epsilon_0 A^{-k},\quad k=1,\dots,K\,,
\end{equation}
where $K$ is determined as the first time $\epsilon_K \le A N^{-1/(2s)}$, see Figure \ref{fig1} for illustration. It is clear that
\begin{equation}\label{Kest}
    \epsilon_K \ge N^{-1/(2s)},\quad K \ge (\frac{1}{2s}-\lambda)\frac{\ln N}{\ln A}-1\ge (\frac{1}{4s}-\lambda)\frac{\ln N}{\ln A}\,,
\end{equation}
for large $N$. Our goal is to choose $\epsilon$ as one of the $\epsilon_k$. We will also denote
\begin{equation}
\epsilon_{k+1/2} = \sqrt{\epsilon_k\epsilon_{k+1}} = \epsilon_0 A^{-(k+1/2)},\quad k=0,\dots,K\,.
\end{equation}

\begin{figure}[htp!]
\begin{center}
	\includegraphics[width=0.99\textwidth]{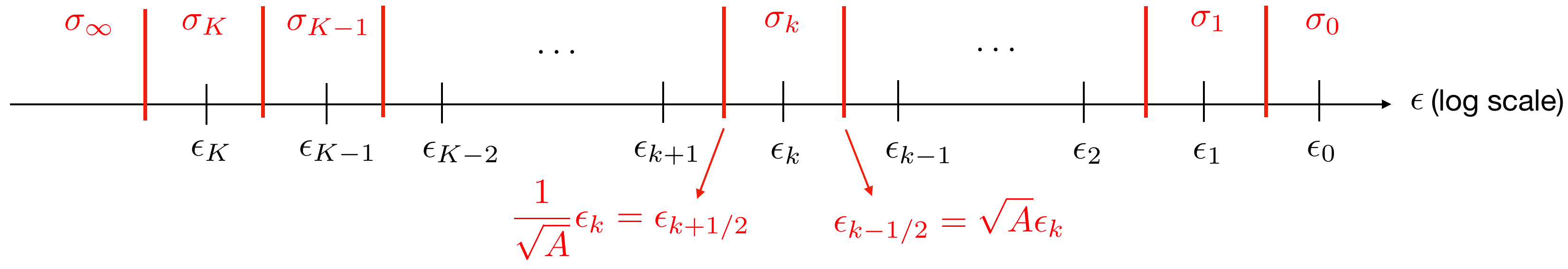}
	\caption{The choice of $\epsilon_k$ in logarithmic scale and their relation with $\sigma_k$. Each $\sigma_k$ (or $\sigma_{k,\epsilon}$) contains terms with $|\bx_i-\bx_j|$ in the range between the two neighboring red segments. One can see that the terms in $\sigma_{l,\epsilon_k}$ has $|\bx_i-\bx_j|\ge \sqrt{A}\epsilon_k$ if $l<k$ and $|\bx_i-\bx_j|\le \frac{1}{\sqrt{A}}\epsilon_k$ if $l>k$.}
	\label{fig1}
\end{center}	
\end{figure}

{\bf STEP 2}: estimate $E^{\epsilon_k}_N(\vec{\bx}) - E_N(\vec{\bx})$ for fixed $k=1,\dots,K$.

For given $\vec{\bx}$, define
\begin{equation}
    \sigma_{k,\epsilon} = \frac{1}{2N(N-1)}\sum_{\epsilon_{k+1/2} \le |\bx_i-\bx_j| < \epsilon_{k-1/2}} W^\epsilon(\bx_i-\bx_j),\quad k=1,\dots,K\,,
\end{equation}
(where the summation happens in a subset of $\{(i,j):1\le i, j \le N;\,i\ne j\}$) and
\begin{equation}\begin{split}
    \sigma_{0,\epsilon} = & \frac{1}{2N(N-1)}\sum_{|\bx_i-\bx_j| \ge \epsilon_{1/2}} W^\epsilon(\bx_i-\bx_j)\,,\\
     \sigma_{\infty,\epsilon} = & \frac{1}{2N(N-1)}\sum_{i\ne j,\,|\bx_i-\bx_j| < \epsilon_{K+1/2}} W^\epsilon(\bx_i-\bx_j)\,.
\end{split}\end{equation}
The counterparts with $W^\epsilon$ replaced by $W$ are denoted by $\sigma_k$. It is clear that for $N$ large, any $\epsilon\le \epsilon_0=N^{-\lambda}$ is small, and thus any $\sigma_{k,\epsilon},\sigma_k,\,k=1,\dots,K,\infty$ only contains positive terms in its summand due to $W(\bx)\sim a_s|\bx|^{-s}$ for small $|\bx|$.

Then, for each $k=1,\dots,K$, we can decompose $E^{\epsilon_k}_N(\vec{\bx})$ into
\begin{equation}
    E^{\epsilon_k}_N(\vec{\bx}) = \sigma_{0,\epsilon_k} + \sum_{l=1}^{k-1}\sigma_{l,\epsilon_k} + \sigma_{k,\epsilon_k} + \sum_{l=k+1}^{K}\sigma_{l,\epsilon_k} + \sigma_{\infty,\epsilon_k}\,,
\end{equation}
and similarly decompose
\begin{equation}
    E_N(\vec{\bx}) = \sigma_{0} + \sum_{l=1}^{k-1}\sigma_{l} + \sigma_{k} + \sum_{l=k+1}^{K}\sigma_{l} + \sigma_{\infty}\,.
\end{equation}
Now we estimate each part $\sigma_{l,\epsilon_k}$ separately by comparing with the counterparts $\sigma_l$. For different values of $l$, we will apply different estimates in Lemma \ref{lem_W} according to the relative size between $\epsilon_k$ and $|\bx_i-\bx_j|$.

{\bf Case $l=0$} (positive contribution): the terms in $\sigma_{0,\epsilon_k}$ have $|\bx_i-\bx_j| \ge \epsilon_{1/2} \ge \sqrt{A}\epsilon_k \ge 2\epsilon_k$ since $k\ge 1$ and $A\ge 4$. Therefore we may apply ({\bf W3}) and use $|\bx_i-\bx_j| \ge \epsilon_{1/2} = A^{(k-1/2)}\epsilon_k$ to get
\begin{equation}\begin{split}
    \sigma_{0,\epsilon_k} - \sigma_0 \le &  C\epsilon_k^2 \frac{1}{2N(N-1)}\sum_{|\bx_i-\bx_j| \ge \epsilon_{1/2}} |\bx_i-\bx_j|^{-s-2} \\
    \le &  CA^{-2(k-1/2)} \frac{1}{2N(N-1)}\sum_{|\bx_i-\bx_j| \ge \epsilon_{1/2}} |\bx_i-\bx_j|^{-s} \\
    \le &  CA^{-2(k-1/2)} \frac{1}{2N(N-1)}\sum_{|\bx_i-\bx_j| \ge \epsilon_{1/2}} (W(\bx_i-\bx_j)+C) \\
    \le &  \cC_EA^{-2(k-1/2)}\,,  \\
\end{split}\end{equation}
where we denote
\begin{equation}\label{CE}
\cC_E = C\max\{E_N(\vec{\bx}),1\}\,,
\end{equation}
with suitable constant $C$. Here the third inequality uses $|\bx|^{-s} \le C(W(\bx)+C)$ due to the fact that $W(\bx) \sim a_s|\bx|^{-s}$ for small $|\bx|$ and smooth elsewhere.

{\bf Case $l=1,\dots,k-1$} (positive contribution): the terms in $\sigma_{l,\epsilon_k}$ have $|\bx_i-\bx_j| \ge \epsilon_{l+1/2} \ge \sqrt{A}\epsilon_k \ge 2\epsilon_k$. Therefore we may apply ({\bf W3}) and use $|\bx_i-\bx_j| \ge \epsilon_{l+1/2} = A^{(k-l-1/2)}\epsilon_k$ to get
\begin{equation}\begin{split}
    \sigma_{l,\epsilon_k} - \sigma_l \le &  C_3\epsilon_k^2 \frac{1}{2N(N-1)}\sum_{\epsilon_{l+1/2} \le |\bx_i-\bx_j| < \epsilon_{l-1/2} } a_s|\bx_i-\bx_j|^{-s-2} \\
    \le &  C_3A^{-2(k-l-1/2)} \frac{1}{2N(N-1)}\sum_{\epsilon_{l+1/2} \le |\bx_i-\bx_j| < \epsilon_{l-1/2} } a_s|\bx_i-\bx_j|^{-s} \\
    \le &  2C_3A^{-2(k-l-1/2)} \frac{1}{2N(N-1)}\sum_{\epsilon_{l+1/2} \le |\bx_i-\bx_j| < \epsilon_{l-1/2} } W(\bx_i-\bx_j) \\
    \le &  2C_3A^{-2(k-l-1/2)} \sigma_l \,,\\
\end{split}\end{equation}
where we used that $W(\bx) \sim a_s|\bx|^{-s}$ for small $|\bx|$ in the third inequality.

{\bf Case $l=k$} (positive contribution): we have $|\bx_i-\bx_j| \le \epsilon_{k-1/2} <\epsilon_0$ since $k\ge 1$, and thus $|\bx_i-\bx_j|$ is small. Therefore we may apply ({\bf W2}) to get
\begin{equation}
    \sigma_{k,\epsilon_k}-\sigma_k \le C_2 \sigma_k-\sigma_k = (C_2-1)\sigma_k\,.
\end{equation}

{\bf Case $l=k+1,\dots,K$} (negative contribution): we may apply ({\bf W1}) and use $|\bx_i-\bx_j| \le \epsilon_{l-1/2} = A^{-(l-k-1/2)}\epsilon_k$ to get
\begin{equation}\begin{split}
    \sigma_{l,\epsilon_k} - \sigma_l \le &  \frac{1}{2N(N-1)}\sum_{\epsilon_{l+1/2} \le |\bx_i-\bx_j| < \epsilon_{l-1/2} } (C_1a_s\epsilon_k^{-s} - W(\bx_i-\bx_j)) \\
    \le &  \frac{1}{2N(N-1)}\sum_{\epsilon_{l+1/2} \le |\bx_i-\bx_j| < \epsilon_{l-1/2} } (C_1a_sA^{-s(l-k-1/2)}|\bx_i-\bx_j|^{-s} - W(\bx_i-\bx_j)) \\
    \le &  \frac{1}{2N(N-1)}\sum_{\epsilon_{l+1/2} \le |\bx_i-\bx_j| < \epsilon_{l-1/2} } (2C_1 A^{-s(l-k-1/2)} W(\bx_i-\bx_j) - W(\bx_i-\bx_j)) \\
    = & -(1-2C_1 A^{-s(l-k-1/2)})\sigma_l\,,
\end{split}\end{equation}
where we used that $W(\bx) \sim a_s|\bx|^{-s}$ for small $|\bx|$ in the third inequality. Notice that the coefficient $1-2C_1 A^{-s(l-k-1/2)} \ge 1-2C_1 A^{-s/2}\ge \frac{1}{2}$ if $A$ satisfies
\begin{equation}
A \ge (4C_1)^{2/s}\,.
\end{equation} 
With this condition, we have 
\begin{equation}
 \sigma_{l,\epsilon_k} - \sigma_l \le -\frac{1}{2}\sigma_l \le 0\,.
\end{equation}

{\bf Case $l=\infty$} (negative contribution): we may proceed similarly as the previous case to get
\begin{equation}
    \sigma_{\infty,\epsilon_k} - \sigma_\infty \le 0\,.
\end{equation}

We sum up these estimates and get
\begin{equation}\label{E1}\begin{split}
    E^{\epsilon_k}_N(\vec{\bx}) - E_N(\vec{\bx})
    \le & \cC_EA^{-2(k-1/2)}  + 2C_3\sum_{l=1}^{k-1} A^{-2(k-l-1/2)} \sigma_l + (C_2-1) \sigma_k - \frac{1}{2}\sum_{l=k+1}^K\sigma_l\,,\\
\end{split}\end{equation}
for any $k=1,\dots,K$.

{\bf STEP 3}: choose $\epsilon$ to be some $\epsilon_k$ based on \eqref{E1}.

If there exists some $k=1,\dots,K-1$ such that the RHS is negative, then we already get \eqref{toprove} with $\epsilon=\epsilon_k$. Otherwise, recalling that $\sigma_1,\dots,\sigma_K\ge 0$ and the inequality $\sum_{k=1}^K\sigma_k\le \cC_E$ by the definition of $\sigma_k$, we may apply Lemma \ref{lem_ind} below with $\alpha=\frac{1}{2}$, $\cC_3=2C_3>0$, $\cC_2=C_2-1>0$ to get $\sigma_K \le 8\cC_E\beta^K$ where $\beta=\frac{C_2-3/4}{C_2-1/2}$, as long as $A$ satisfies
\begin{equation}
A \ge \max\{4,64C_3\}\,.
\end{equation}
Taking $\epsilon=\epsilon_K$, we get from \eqref{E1} that
\begin{equation}\begin{split}
    E^{\epsilon_K}_N(\vec{\bx}) - E_N(\vec{\bx}) 
    \le & \cC_EA^{-2K+1}  + 2C_3\sum_{l=1}^{K-1} A^{-2(K-l)+1} \cdot 8\cC_E\beta^l + (C_2-1)\cdot 8\cC_E \beta^K \\
    \le &  C \cC_E \beta^K\,, \\
\end{split}\end{equation}
using $A\ge 4$ and $\frac{1}{2}<\beta<1$. Notice that \eqref{Kest} implies
\begin{equation}
    \ln(\beta^K) = K\ln\beta \le (\frac{1}{4s}-\lambda)\frac{\ln \beta}{\ln A} \ln N\,,
\end{equation}
and then we see that
\begin{equation}\label{Npower}
    \beta^K \le N^{(\frac{1}{4s}-\lambda)\frac{\ln \beta}{\ln A}}\,.
\end{equation}
Taking $\lambda>0$ so that $-\lambda=(\frac{1}{4s}-\lambda)\frac{\ln \beta}{\ln A}$ (noticing $-1<\frac{\ln \beta}{\ln A}<0$), we get the conclusion.

\begin{lemma}\label{lem_ind}
Let $\alpha>0$. Assume $\sigma_1,\dots,\sigma_K\ge 0$ satisfy $\sum_{k=1}^K\sigma_k\le \cC_E$ and
\begin{equation}\begin{split}
    \cC_EA^{-2(k-1/2)}  + \cC_3\sum_{l=1}^{k-1} A^{-2(k-l-1/2)} \sigma_l + \cC_2 \sigma_k - \alpha \sum_{l=k+1}^K\sigma_l \ge 0\,,
\end{split}\end{equation}
for any $k=1,\dots,K-1$, where $\cC_E,\cC_3,\cC_2,A>0$ are given constants. If 
\begin{equation}\label{A}
A \ge \max\{4,\frac{16\cC_3}{\alpha}\}\,,
\end{equation}
then we have the estimate
\begin{equation}\label{lem_ind_1}
    \sigma_k \le \frac{4}{\alpha}\cC_E \beta^k,\quad k=1,\dots,K\,,
\end{equation}
where $\beta=\frac{\cC_2+\frac{\alpha}{2}}{\cC_2+\alpha}<1$.
\end{lemma}

\begin{proof}
Define $S_k = \sum_{l=k}^K\sigma_l \ge 0 $. Then the assumption gives
\begin{equation}\begin{split}
    \alpha S_{k+1} \le & \cC_EA^{-2k+1}  + \cC_3\sum_{l=1}^{k-1} A^{-2(k-l)+1} \sigma_l +\cC_2 \sigma_k \\
    = & \cC_EA^{-2k+1}  + \cC_3\sum_{l=1}^{k-1} A^{-2(k-l)+1} (S_l-S_{l+1}) + \cC_2 (S_k-S_{k+1}) \\
    \le & \cC_EA^{-2k+1}  + \cC_3\sum_{l=1}^{k-1} A^{-2(k-l)+1} S_l + \cC_2 (S_k-S_{k+1}) \,,\\
\end{split}\end{equation}
for any $k=1,\dots,K-1$, i.e.,
\begin{equation}\label{sk1}\begin{split}
    S_{k+1} 
    \le & \frac{\cC_E}{\cC_2+\alpha}A^{-2k+1}  + \frac{\cC_3}{\cC_2+\alpha}\sum_{l=1}^{k-1} A^{-2(k-l)+1} S_l + \frac{\cC_2}{\cC_2+\alpha} S_k\,. \\
\end{split}\end{equation}

Then, for sufficiently large $A$, we prove
\begin{equation}\label{indk}
    S_k \le C_4 \cC_E \beta^k,\quad k = 1,\dots,K\,,
\end{equation}
by induction, where $C_4>2$ and $1/2<\beta<1$ are to be determined. The case $k=1$ is true since $C_4\beta \ge 1$ and $S_1 =\sum_{k=1}^K\sigma_k\le \cC_E$. Assuming \eqref{indk} is true for $1,\dots,k$. Then \eqref{sk1} gives
\begin{equation}\label{S3terms}\begin{split}
    \frac{S_{k+1}}{C_4\cC_E} 
    \le & \frac{1}{(\cC_2+\alpha)C_4}A^{-2k+1}  + \frac{\cC_3}{\cC_2+\alpha}\sum_{l=1}^{k-1} A^{-2(k-l)+1} \beta^l + \frac{\cC_2}{\cC_2+\alpha} \beta^k \\
    \le & \frac{1}{(\cC_2+\alpha)C_4}A^{-k}  + \frac{\cC_3}{\cC_2+\alpha}\sum_{l=1}^{k-1} A^{-(k-l)} \beta^l + \frac{\cC_2}{\cC_2+\alpha} \beta^k \\
    \le & \frac{1}{(\cC_2+\alpha)C_4}\beta^k  + \frac{\cC_3}{\cC_2+\alpha}\beta^k \sum_{l=1}^{k-1} (A\beta)^{-(k-l)}  + \frac{\cC_2}{\cC_2+\alpha} \beta^k \\
    \le & \beta^k\Big(\frac{1}{(\cC_2+\alpha)C_4}  + \frac{\cC_3}{\cC_2+\alpha} (A\beta)^{-1}\frac{1}{1-(A\beta)^{-1}}  + \frac{\cC_2}{\cC_2+\alpha} \Big)\,, \\
\end{split}\end{equation}
where we used the fact that $A^{-1} < \beta$ (since $A\ge 4$ by \eqref{A} and $\beta>1/2$). We choose $\beta=\frac{\cC_2+\frac{\alpha}{2}}{\cC_2+\alpha}$ which satisfies $1/2<\beta<1$. Then, by taking 
\begin{equation}
C_4=\frac{4}{\alpha}\,,
\end{equation}
and requiring \eqref{A}, the first two terms in the parenthesis can both be made smaller than $\frac{\alpha/4}{\cC_2+\alpha}$, and thus the parenthesis is smaller than $\beta$. This finishes the induction step.

Finally, notice that $\eqref{indk}$ implies the conclusion.

\end{proof}

\begin{remark}[Dependence of $\lambda$ on other parameters]\label{rem_lambda}
The parameter $\lambda$ in the proof of Theorem \ref{thm1} was determined by the following procedure:
\begin{itemize}
\item Obtain $C_1,C_2,C_3$ in Lemma \ref{lem_W} which depend on $d$ and $s$.
\item Determine $A$ by $A=\max\{4, 64C_3,(4C_1)^{2/s}\}$.
\item Determine $\lambda$ by the relation $-\lambda=(\frac{1}{4s}-\lambda)\frac{\ln \beta}{\ln A}$, where $\beta=\frac{C_2-3/4}{C_2-1/2}<1$.
\end{itemize}
As a consequence, $\lambda$ depends on $d$ and $s$ and can be calculated explicitly.
\end{remark}

\section{Proof of Theorem \ref{thm1} when $s=0$ (logarithmic potential)}\label{sec_thm1s0}

In this section we prove Theorem \ref{thm1} when $s=0$, which is formulated as a stronger version as follows:
\begin{theorem}\label{thm1log}
When $s=0$, Theorem \ref{thm1} holds with 
\begin{equation}
    d_\infty(\rho,\rho_N[\vec{\bx}]) \le N^{-1}\,,
\end{equation}
and 
\begin{equation}
    E[\rho] - E_N(\vec{\bx}) \le C N^{-1}\ln^2 N\max\{E_N(\vec{\bx}),1\}\,.
\end{equation}
\end{theorem}

\begin{proof}
The proof follows the same strategy as the previous section, with Lemma \ref{lem_Wlog} used instead of Lemma \ref{lem_W}. We first notice that by taking the same $\rho$ as before, if we take $\epsilon$ sufficiently small with
\begin{equation}
\epsilon \ge N^{-M\ln N}\,,
\end{equation}
(where $M$ is a fixed large number to be chosen) then we have 
\begin{equation}\begin{split}
    E^\epsilon[\rho_N] - E_N^\epsilon(\vec{\bx}) 
    \le & \frac{1}{2N}\max W^\epsilon + \frac{1}{2N^2(N-1)}\cdot N(N-1)\max |W^\epsilon| \\
     \le &  C N^{-1} (-\ln \epsilon) \le C M N^{-1}\ln^2 N\,,
\end{split}\end{equation}
using ({\bf W1}).

We choose $\epsilon_k$ similar as before, with (we have implicitly assumed that $N$ is sufficiently large)
\begin{equation}
\epsilon_0=N^{-1},\quad A=N\,,
\end{equation}
(i.e., $\epsilon_k=N^{-(k+1)}$) and $K$ is determined as the first time $\epsilon_K \le AN^{-M\ln N}$. Notice here that $A$ depends on $N$, which is different from the previous section. Then we have $\epsilon_K\ge N^{-M\ln N}$ and
\begin{equation}\label{Kestlog}
K \ge (M\ln N-1)\frac{\ln N}{\ln A}-1=M\ln N-2\,.
\end{equation}
Then, for each $k=1,\dots,K$, we can decompose $E^{\epsilon_k}_N(\vec{\bx})$ into
\begin{equation}
    E^{\epsilon_k}_N(\vec{\bx}) = \sigma_{0,\epsilon_k} + \sum_{l=1}^{k-1}\sigma_{l,\epsilon_k} + \sigma_{k,\epsilon_k} + \sum_{l=k+1}^{K}\sigma_{l,\epsilon_k} + \sigma_{\infty,\epsilon_k}\,,
\end{equation}
as before. Now we estimate each part separately by comparing with the counterparts with the potential $W$.

{\bf Case $l=0$}: 
apply ({\bf W3}) to get
\begin{equation}\begin{split}
    \sigma_{0,\epsilon_k} - \sigma_0 \le &  C\epsilon_k^2 \frac{1}{2N(N-1)}\sum_{|\bx_i-\bx_j| \ge \epsilon_{1/2}} |\bx_i-\bx_j|^{-2} \\
    \le &  CA^{-2(k-1/2)}  \frac{1}{2N(N-1)}\sum_{|\bx_i-\bx_j| \ge \epsilon_{1/2}} 1 \\
   \le &  CA^{-2(k-1/2)}  \frac{1}{2N(N-1)}\sum_{|\bx_i-\bx_j| \ge \epsilon_{1/2}} (W(\bx_i-\bx_j)+C) \\
    \le &  \cC_E A^{-2(k-1/2)} \,,
\end{split}\end{equation}
where $\cC_E$ is given by \eqref{CE}.

{\bf Case $l=1,\dots,k-1$}: 
apply ({\bf W3}) 
to get
\begin{equation}\begin{split}
    \sigma_{l,\epsilon_k} - \sigma_l \le &  C_3\epsilon_k^2 \frac{1}{2N(N-1)}\sum_{\epsilon_{l+1/2} \le |\bx_i-\bx_j| < \epsilon_{l-1/2} } a_0|\bx_i-\bx_j|^{-2} \\
    \le &  C_3A^{-2(k-l-1/2)} \frac{1}{2N(N-1)}\sum_{\epsilon_{l+1/2} \le |\bx_i-\bx_j| < \epsilon_{l-1/2} } a_0 \\
    \le &  2C_3A^{-2(k-l-1/2)} \frac{1}{2N(N-1)}\sum_{\epsilon_{l+1/2} \le |\bx_i-\bx_j| < \epsilon_{l-1/2} } W(\bx_i-\bx_j)\cdot \frac{1}{-\ln \epsilon_{l-1/2}} \\
    \le &  2C_3A^{-2(k-l-1/2)} \frac{1}{l}\sigma_l \,,\\
\end{split}\end{equation}
where we used that $W(\bx) \sim -a_0\ln|\bx|$ for small $|\bx|$, and $-\ln \epsilon_{l-1/2} = (l+\frac{1}{2})\ln N \ge l$.

{\bf Case $l=k$}: apply ({\bf W2}) to get
\begin{equation}\begin{split}
    \sigma_{k,\epsilon_k} \le & \frac{1}{2N(N-1)}\sum_{\epsilon_{k+1/2} \le |\bx_i-\bx_j| < \epsilon_{k-1/2} } (W(\bx_i-\bx_j)+C_2a_0) \\
    = & \sigma_k + C_2\frac{1}{2N(N-1)}\sum_{\epsilon_{k+1/2} \le |\bx_i-\bx_j| < \epsilon_{k-1/2} } a_0\\
    \le & \sigma_k + C_2\frac{1}{2N(N-1)}\sum_{\epsilon_{k+1/2} \le |\bx_i-\bx_j| < \epsilon_{k-1/2} } W(\bx_i-\bx_j)\cdot \frac{1}{-\ln \epsilon_{k-1/2}}\\
    \le & \sigma_k + C_2\frac{1}{k}\sigma_k\,,
\end{split}\end{equation}
where in the last inequality we used $-\ln \epsilon_{k-1/2}\ge k$ as the previous case.

{\bf Case $l=k+1,\dots,K$}: apply ({\bf W1}) to get
\begin{equation}\begin{split}
    \sigma_{l,\epsilon_k} - \sigma_l \le &  \frac{1}{2N(N-1)}\sum_{\epsilon_{l+1/2} \le |\bx_i-\bx_j| < \epsilon_{l-1/2} } (-a_0\ln \epsilon_k + C_1 - W(\bx_i-\bx_j)) \\
    \le &  \frac{1}{2N(N-1)}\sum_{\epsilon_{l+1/2} \le |\bx_i-\bx_j| < \epsilon_{l-1/2} } (-a_0\ln \epsilon_k + C_1 + a_0\ln\epsilon_{l-1/2} + C_1^*) \\
    \le &  \frac{1}{2N(N-1)}\sum_{\epsilon_{l+1/2} \le |\bx_i-\bx_j| < \epsilon_{l-1/2} } ((C_1+C_1^*) - a_0(l-k-\frac{1}{2})\ln A) \,, \\
\end{split}\end{equation}
where we used that $\big|W(\bx) +a_0\ln|\bx|\big| \le C_1^*$ for small $|\bx|$ and some constant $C_1^*$. Notice that $l-k-\frac{1}{2}\ge \frac{1}{2}$. Therefore, using
$
A=N > e^{4(C_1+C_1^*)/a_0}
$ 
for sufficiently large $N$, the quantity $((C_1+C_1^*) - a_0(l-k-\frac{1}{2})\ln A)$ in the above inequality is less than $- \frac{1}{2}a_0(l-k-\frac{1}{2})\ln A<0$, and we may proceed as
\begin{equation}\begin{split}
    \sigma_{l,\epsilon_k} - \sigma_l \le &   - \frac{1}{2}a_0(l-k-\frac{1}{2})\ln A\frac{1}{2N(N-1)}\sum_{\epsilon_{l+1/2} \le |\bx_i-\bx_j| < \epsilon_{l-1/2} } 1 \\
    \le &   - \frac{1}{4}a_0(l-k-\frac{1}{2})\ln A\frac{1}{2N(N-1)}\sum_{\epsilon_{l+1/2} \le |\bx_i-\bx_j| < \epsilon_{l-1/2} } \frac{W(\bx_i-\bx_j)}{- a_0\ln \epsilon_{l+1/2}} \\
    = & -\frac{(l-k-\frac{1}{2})\ln A}{4(-\ln \epsilon_0 + (l+\frac{1}{2})\ln A)}\sigma_l \\
    \le &  -\frac{\frac{1}{2}\ln A}{4(l+\frac{3}{2})\ln A}\sigma_l \\
    \le & -\frac{1}{16l}\sigma_l\,,
\end{split}\end{equation}
using $l-k-\frac{1}{2}\ge \frac{1}{2}$ and $l\ge 2$ respectively in the last two inequalities.

{\bf Case $l=\infty$}: apply ({\bf W1}) to get
\begin{equation}
    \sigma_{\infty,\epsilon_k} - \sigma_\infty \le 0\,,
\end{equation}
similar as the previous case.

We summarize these estimates and get
\begin{equation}\label{E2}\begin{split}
    E^{\epsilon_k}_N(\vec{\bx}) - E_N(\vec{\bx}) \le & \cC_EA^{-2(k-1/2)} + 2C_3\sum_{l=1}^{k-1}A^{-2(k-l-1/2)} \frac{1}{l}\sigma_l + C_2\frac{1}{k}\sigma_k -\frac{1}{16}\sum_{l=k+1}^K\frac{1}{l}\sigma_l \,,\\
\end{split}\end{equation}
for $k=1,\dots,K$.

Notice that $A=N$ is sufficiently large. Applying Lemma \ref{lem_ind} with $\alpha=\frac{1}{16}$ to the sequence $\{\sigma_k/k\}_{k=1}^K$, we see that either the RHS of \eqref{E2} is negative for some $k$, or $\sigma_k/k\le C\cC_E\beta^k,\,\forall k=1,\dots,K$ (with $\beta=\frac{C_2+1/32}{C_2+1/16}$). In the former case, we find an $\epsilon=\epsilon_k\le \epsilon_0$ with $E^{\epsilon_k}_N(\vec{\bx}) \le E_N(\vec{\bx})$. 

In the latter case, taking $\epsilon=\epsilon_K$, we get 
\begin{equation}
E^{\epsilon_K}_N(\vec{\bx}) - E_N(\vec{\bx}) \le C\cC_E\beta^K \le C\cC_E  \beta^{M \ln N} = C \cC_E N^{M \ln \beta}\,,
\end{equation}
using \eqref{Kestlog}. By taking $M=-\frac{1}{\ln \beta}$ we get
\begin{equation}
E^{\epsilon_K}_N(\vec{\bx}) - E_N(\vec{\bx}) \le  C  N^{-1} \,,
\end{equation}
which finishes the proof.
\end{proof}

\section{Acknowledgement}

The author would like to thank Jiuya Wang for proposing the questions considered in this paper and helpful conversations. The author would like to thank Sylvia Serfaty for communications on the existing results on mean field limits. The author would also like to thank Zhenfu Wang for helpful conversations.

\section{Appendix: proof of Lemma \ref{lem_min}}\label{app_min}

We first prove $\min E_N\le \min E_{N+1}$ which implies that $\min E_{N_1}\le \min E_{N_2}$ for any $N_1<N_2$. Let $\vec{\bx}= (\bx_1,\dots,\bx_i,\dots,\bx_{N+1})$ be a minimizer of $E_{N+1}$, and denote $\vec{\bx}_{\hat{i}} = (\bx_1,\dots,\hat{\bx}_i,\dots,\bx_{N+1})$ to be the $N$ particle configuration by deleting the $i$-th particle. Then
\begin{equation}
\min E_N \le E_N(\vec{\bx}_{\hat{i}}) = \frac{1}{2N(N-1)}\sum_{j,k=1,\,j\ne k,\,j\ne i,\,k\ne i}^{N+1} W(\bx_j-\bx_k)\,.
\end{equation}
Sum in $i=1,\dots,N+1$, we see that every pair $(j,k),\,j\ne k$ appears $N-1$ times on the RHS. Therefore
\begin{equation}\begin{split}
(N+1)\min E_N \le & \sum_{i=1}^{N+1} E_N(\vec{\bx}_{\hat{i}}) = \frac{1}{2N(N-1)}\cdot (N-1)\sum_{j,k=1,\,j\ne k}^{N+1} W(\bx_j-\bx_k) \\
= & (N+1)E_{N+1}(\vec{\bx}) = (N+1)\min E_{N+1}\,,
\end{split}\end{equation}
which gives $\min E_N\le \min E_{N+1}$.

Then we prove $\min E_N \le \min E$. Let $\rho$ be a minimizer of $E$. We integrate the equation
\begin{equation}
E_N(\vec{\bx}) = \frac{1}{2N(N-1)}\sum_{i,j=1,\, i\ne j}^N W(\bx_i-\bx_j)\,,
\end{equation}
with respect to the measure $\rho(\bx_1)\rd{\bx_1}\dots\rho(\bx_N)\rd{\bx_N}$ (which is a probability measure on $(\mathbb{T}^d)^N$), and get
\begin{equation}\begin{split}
& \int_{(\mathbb{T}^d)^N}E_N(\vec{\bx})\rho(\bx_1)\rd{\bx_1}\dots\rho(\bx_N)\rd{\bx_N} \\
= & \frac{1}{2N(N-1)}\sum_{i,j=1,\, i\ne j}^N \int_{(\mathbb{T}^d)^N}W(\bx_i-\bx_j)\rho(\bx_1)\rd{\bx_1}\dots\rho(\bx_N)\rd{\bx_N} \\
= & \frac{1}{2N(N-1)}\sum_{i,j=1,\, i\ne j}^N \int_{(\mathbb{T}^d)^N}W(\bx_i-\bx_j)\rho(\bx_i)\rd{\bx_i}\rho(\bx_j)\rd{\bx_j} \\
= & \frac{1}{N(N-1)}\sum_{i,j=1,\, i\ne j}^N E[\rho] \\
= & \min E\,.
\end{split}\end{equation}
Since the integral of $E_N(\vec{\bx})$ against a probability measure is $\min E$, there exists a point $\vec{\bx}\in (\mathbb{T}^d)^N$ such that $E_N(\vec{\bx}) \le \min E$. Therefore $\min E_N \le \min E$.

Finally we prove $\lim_{N\rightarrow\infty}\min E_N = \min E$. First, since we have proved that $\{\min E_N\}$ is increasing in $N$ and bounded above by $\min E$, the limit $\lim_{N\rightarrow\infty}\min E_N=: L$ exists and $L\le \min E$. Now we define $E_N^m$ as the discrete interaction energy with potential $\min\{W(\bx),m\}$, and similarly define $E^m$. 

For each $N$, take $\vec{\bx}_N$ as a minimizer of $E_N$, and $\rho_N$ its empirical measure. Take a weakly convergent subsequence $\{\rho_{N_k}\}$ which converges to $\rho$. Then
\begin{equation}\begin{split}
E^m[\rho] \le & \liminf_{k\rightarrow\infty}E^m[\rho_{N_k}]  \\
= & \liminf_{k\rightarrow\infty}(\frac{N_k-1}{N_k}E_{N_k}^m(\vec{\bx}_{N_k})+\frac{1}{2N_k}m) \\
\le & \liminf_{k\rightarrow\infty}(\frac{N_k-1}{N_k}E_{N_k}(\vec{\bx}_{N_k})+\frac{1}{2N_k}m) \\
= & L\,,
\end{split}\end{equation} 
where the first inequality uses the lower-semicontinuity of the functional $E^m$. Then taking $m\rightarrow\infty$, by the monotone convergence theorem we get
\begin{equation}
\min E \le E[\rho] = \lim_{m\rightarrow\infty}E^m[\rho] \le L\,.
\end{equation}
Combined with the inequality $L\le \min E$, we see that $L=\min E$.

\bibliographystyle{alpha}
\bibliography{set1.bib}

\newcommand{\etalchar}[1]{$^{#1}$}
\begin{thebibliography}{BCLR13b}

\bibitem[BCLR13a]{Ca13}
D.~Balagu{\'e}, J.~A. Carrillo, T.~Laurent, and G.~Raoul.
\newblock Dimensionality of local minimizers of the interaction energy.
\newblock {\em Archive for Rational Mechanics and Analysis}, 209(3):1055--1088,
  2013.

\bibitem[BCLR13b]{Ca13_2}
D.~Balagu{\'e}, J.~A. Carrillo, T.~Laurent, and G.~Raoul.
\newblock Nonlocal interactions by repulsive-attractive potentials: radial
  ins/stability.
\newblock {\em Phys. D.}, 260:5--25, 2013.

\bibitem[BCT18]{BCT}
A.~Burchard, R.~Choksi, and I.~Topaloglu.
\newblock Nonlocal shape optimization via interactions of attractive and
  repulsive potentials.
\newblock {\em Indiana Univ. Math. J.}, 67(1):375--395, 2018.

\bibitem[BHS19]{saffbook2}
S.~V. Borodachov, D.~P. Hardin, and E.~B. Saff.
\newblock {\em Discrete energy on rectifiable sets}.
\newblock Springer, 2019.

\bibitem[BKS{\etalchar{+}}15]{BKSUB}
A.~L. Bertozzi, T.~Kolokolnikov, H.~Sun, D.~Uminsky, and J.~von Brecht.
\newblock Ring patterns and their bifurcations in a nonlocal model of
  biological swarms.
\newblock {\em Commun. Math. Sci.}, 13(4):955--985, 2015.

\bibitem[CCP15]{CCP15}
J.~A. Ca{\~{n}}izo, J.~A. Carrillo, and F.~S. Patacchini.
\newblock Existence of compactly supported global minimisers for the
  interaction existence of compactly supported global minimisers for the
  interaction energy.
\newblock {\em Arch. Ration. Mech. Anal.}, 217(3):1197--1217, 2015.

\bibitem[CDM16]{CDM}
J.~A. Carrillo, M.~G. Delgadino, and A.~Mellet.
\newblock Regularity of local minimizers of the interaction energy via obstacle
  problems.
\newblock {\em Comm. Math. Phys.}, 343(3):747--781, 2016.

\bibitem[CFP17]{CFP17}
J.~A. Carrillo, A.~Figalli, and F.~S. Patacchini.
\newblock Geometry of minimizers for the interaction energy with mildly
  repulsive potentials.
\newblock {\em Ann. IHP}, 34:1299--1308, 2017.

\bibitem[CHM18]{CHM}
D.~Chafa\"i, A.~Hardy, and M.~Ma\"ida.
\newblock Concentration for coulomb gases and coulomb transport inequalities.
\newblock {\em Journal of Functional Analysis}, 275:1447--1483, 2018.

\bibitem[CP18]{CP18}
J.~A. Ca{\~{n}}izo and F.~S. Patacchini.
\newblock Discrete minimisers are close to continuum minimisers for the
  interaction energy.
\newblock {\em Calculus of Variations and Partial Differential Equations},
  57(1):24, 2018.

\bibitem[CS21]{ShuCar}
J.~A. Carrillo and R.~Shu.
\newblock From radial symmetry to fractal behavior of aggregation equilibria
  for repulsive-attractive potentials.
\newblock {\em arXiv preprint arXiv:2107.05079}, 2021.

\bibitem[DLM22]{DLM1}
C.~Davies, T.~Lim, and R.~J. McCann.
\newblock Classifying minimum energy states for interacting particles:
  spherical shells.
\newblock {\em SIAM J. Appl. Math.}, 82(4):1520--1536, 2022.

\bibitem[DLM23]{DLM2}
C.~Davies, T.~Lim, and R.~J. McCann.
\newblock Classifying minimum energy states for interacting particles: Regular
  simplices.
\newblock {\em Comm. Math. Phys.}, 399(2):577--598, 2023.

\bibitem[ET50]{ET}
P.~Erd\H{o}s and P.~Tur\'{a}n.
\newblock On the distribution of roots of polynomials.
\newblock {\em Ann. of Math. (2)}, 51:105--119, 1950.

\bibitem[FM]{FM}
R.~L. Frank and R.~W. Matzke.
\newblock Minimizers for an aggregation model with attactive-repulsive
  interaction.
\newblock {\em preprint, arXiv:2307.13769}.

\bibitem[Fra22]{Fra}
R.~L. Frank.
\newblock Minimizers for a one-dimensional interaction energy.
\newblock {\em Nonlinear Analysis}, 216:112691, 2022.

\bibitem[HSSS17]{HSSS}
D.~P. Hardin, E.~B. Saff, B.~Z. Simanek, and Y.~Su.
\newblock Next order energy asymptotics for {R}iesz potentials on flat tori.
\newblock {\em International Mathematics Research Notices}, 12:3529--3556,
  2017.

\bibitem[KSUB11]{KSUB}
T.~Kolokolnikov, H.~Sun, D.~Uminsky, and A.~L. Bertozzi.
\newblock Stability of ring patterns arising from two-dimensional particle
  interactions.
\newblock {\em Phys. Rev. E}, 84:015203, Jul 2011.

\bibitem[Lop19]{Lop19}
O.~Lopes.
\newblock Uniqueness and radial symmetry of minimizers for a nonlocal
  variational problem.
\newblock {\em Commun. Pure Appl. Anal.}, 18(5):2265--2282, 2019.

\bibitem[LS17]{LS}
J.~Lu and S.~Steinerberger.
\newblock {R}iesz energy on the torus: Regularity of minimizers.
\newblock {\em preprint, arXiv:1710.08010}, 2017.

\bibitem[LS18]{LS18}
T.~Lebl'e and S.~Serfaty.
\newblock Fluctuations of two dimensional {C}oulomb gases.
\newblock {\em Geometric and Functional Analysis}, 28(2):443--508, 2018.

\bibitem[NRS22]{NRS}
Q.-H. Nguyen, M.~Rosenzweig, and S.~Serfaty.
\newblock Mean-field limits of {R}iesz-type singular flows.
\newblock {\em Ars Inveniendi Analytica}, (4):45, 2022.

\bibitem[Ser20]{Ser20b}
S.~Serfaty.
\newblock Mean field limit for {C}oulomb-type flows.
\newblock {\em Duke Math. J.}, 169(15):2887--2935, 10 2020.

\bibitem[SST15]{SST15}
R.~Simione, D.~Slep\u{c}ev, and I.~Topaloglu.
\newblock Existence of ground states of nonlocal interaction energies.
\newblock {\em J. Stat. Phys.}, 159(4):972--986, 2015.

\bibitem[ST21]{ST}
R.~Shu and E.~Tadmor.
\newblock Newtonian repulsion and radial confinement: convergence towards
  steady state.
\newblock {\em Mathematical Models and Methods in Applied Sciences}, pages
  1--25, 2021.

\bibitem[SWa]{SW21GET}
R.~Shu and J.~Wang.
\newblock Generalized {E}rd{\H{o}}s--{T}ur{\'a}n inequalities and stability of
  energy minimizers.
\newblock {\em preprint, arXiv:2110.03019}.

\bibitem[SWb]{SW21}
R.~Shu and J.~Wang.
\newblock The sharp {E}rd{\H{o}}s--{T}ur{\'a}n inequality.
\newblock {\em arXiv preprint arXiv:2109.11006}.

\end{thebibliography}

\end{document}